\documentclass[]{amsart}

\usepackage[utf8]{inputenc}
\usepackage[OT2,T1]{fontenc}
\DeclareSymbolFont{cyrletters}{OT2}{wncyr}{m}{n}
\DeclareMathSymbol{\Sha}{\mathalpha}{cyrletters}{"58}
\usepackage{graphicx}
\graphicspath{ {./images/} }

\usepackage[hyphens,spaces,obeyspaces]{url}
\usepackage[colorlinks,allcolors=blue,hyperindex,breaklinks]{hyperref}
\hypersetup{
           breaklinks=true,   
           colorlinks=true,   
        }

\usepackage{orcidlink}

\usepackage{pgfplots}
\pgfplotsset{compat=1.18}

\usepackage{mathrsfs}

\title[]{Convergent series of coherent states}
\date{\today}
\usepackage[foot]{amsaddr}
\usepackage{comment}

\author[Peter Vang Uttenthal]{Peter Vang Uttenthal \orcidlink{0009-0001-0878-8213}}
\address{Department of Mathematics, Aarhus Universitet, Ny Munkegade 118, 1530-421, DK-8000
Aarhus C, Denmark}
\email{petervang@math.au.dk}


\newcommand{\Q}{\mathbb{Q}}
\newcommand{\Z}{\mathbb{Z}}

\newcommand{\R}{\mathbb{R}}
\newcommand{\cc}{\mathbb{C}}

\usepackage{amsthm,amsmath, mathrsfs, mathtools}
\usepackage{amsfonts}
\usepackage{amssymb}
\usepackage{fancyhdr}
\usepackage{IEEEtrantools}
\usepackage{tikz-cd}
\usepackage[english]{babel}
\usepackage[utf8]{inputenc}
\usepackage{csquotes}

\newtheorem{theorem}{Theorem}
\numberwithin{theorem}{section}
\newtheorem{lemma}{Lemma}
\numberwithin{lemma}{section}

\numberwithin{definition}{section}
\newtheorem{proposition}{Proposition}
\numberwithin{proposition}{section}
\newtheorem{corollary}{Corollary}
\numberwithin{corollary}{section}
\newtheorem{remark}{Remark}
\numberwithin{remark}{section}

\numberwithin{conjecture}{section}

\numberwithin{question}{section}

\numberwithin{example}{section}

\usepackage[hyphens,spaces,obeyspaces]{url}
\usepackage[colorlinks,allcolors=blue,hyperindex,breaklinks]{hyperref}

\hypersetup{
           breaklinks=true,   
           colorlinks=true,   
        }
\usepackage{tabularx}
\usepackage{subcaption}
\begin{document}

\maketitle
\begin{abstract}
For a discrete subspace $\Gamma$ of a homogeneous space $X$ and a positive definite reproducing kernel Hilbert space
$\mathscr{H}(X)$, a convergence proof is given for
series of partial Whitney functions over $\Gamma$
that arise in the span of coherent states in 
$\mathscr{H}(X)$.
\end{abstract} 
\tableofcontents

\section{Introduction}
Let $X$ be a homogeneous space and let $f \in \mathscr{C}^\infty(X)$ be a smooth function on $X$. 
Suppose that a discrete subspace
$\Gamma = \{\gamma_k \in X: k\geq 1\}$ 
and a sequence 
$(W_n)_n$ in $\mathscr{C}^\infty(X)$ 
are given. 
If $W_n - f$
vanishes on $\Gamma_n:=\{\gamma_k \in \Gamma: 1 \leq k \leq n\}$ for every $n$,
does
$$
W_n \xlongrightarrow{} f
$$
pointwise at every $z\in X$ as $n$ tends to infinity? A smooth function on $X$ 
that agrees with $f$ on $\Gamma_n$ 
is called a Whitney function for $f|_{\Gamma_n}$, 
denoted $W(f|_{\Gamma_n})$.
In this work, we give affirmative answers 
to the question for Whitney functions
that arise in the linear span of coherent states 
in a positive definite reproducing kernel Hilbert space of functions on $X$.

For an analytic function $f$ with given 
values at equidistant points 
$\Gamma_n = \{\gamma_k \in \pi \Z: |k| \leq n\}$ for all $n$,
Whittaker defined the cardinal series  
\begin{equation*} 
C(f) = \lim_{n \to \infty}
\sum_{\gamma_k \in \Gamma_n} a_k e_{\gamma_k}, 
\quad e_{\gamma_k}(z) = \operatorname{sinc}(z-\gamma_k)
\end{equation*}
and initiated the study of its asymptotic properties \cite{Whittaker}; 
the Whittaker-Shannon theorem states that 
$C(f) (z)= f(z)$ for every $z$ if $f$ is a Payley-Wiener function \cite{interpolation}.
For a fixed index $k$, notice in the cardinal series that $a_k$ remains constant for all $n >|k|$.
In contrast, for a sequence of
Whitney functions 
\[
W(f|_{\Gamma_n})
= \sum_{ \gamma_k \in \Gamma_n} a_{k,n} e_{\gamma_k}
\] 
and a fixed index $k$, the coefficients in 
$\{a_{k,n} : n\geq k\}$
depend on $n$.
As a result, the asymptotic behavior of the sequence  
$(W(f|_{\Gamma_n}))_n$ is difficult 
to control.
The unstable coefficients underlie  
the need 
for 
at least one equidistant $\Gamma$--coordinate 
in the convergence proof in Section \ref{sec:spheres} with classical Fourier analytic techniques.

In contrast,
the only assumption 
in Section \ref{sec:coherent} is that $\Gamma$ is sufficiently dispersed in the ambient space $X$. Note that if $X=G/H$ for reductive Lie groups $G$ and $H$ then $X$ is second-countable,
so every discrete subspace of $X$ is countable.

\begin{theorem}[cf. Theorem \ref{conv}] \label{thmintro} Let $X$ be a homogeneous space with respect to reductive Lie group $G$,
and let $\mathscr{H}(X)$ be a Hilbert space of functions on $X$
admitting a positive definite reproducing kernel 
\[
K: X\times X \to \mathbb{C}.
\]
Let $\Gamma$ be discrete topological subspace of $X$ and define $\Gamma_n = \{\gamma_k \in \Gamma: 1\leq k\leq n\}$ with respect to a fixed enumeration of $\Gamma$. 
If the linear span of the coherent states $(e_\gamma)_{\gamma \in \Gamma}$
is dense in  $\mathscr{H}(X)$, then for every smooth $f\in \mathscr{H}(X)$, the convergence 
\[
W(f|_{\Gamma_n}) \xlongrightarrow{} f
\]
holds in the norm topology and pointwise at every $z\in X$ as $n$ tends to infinity.
\end{theorem}
The new input is to formulate the convergence problem in a Hilbert space with a positive definite reproducing kernel, so that
$W(f|_{\Gamma_n})$ 
becomes  
the orthogonal projection of $f$ 
onto the span of the coherent states
associated with $\Gamma_n$.
The difficulties from Section \eqref{sec:spheres} do not enter, allowing us to prove
convergence for general $\Gamma$ in Theorem \ref{thmintro}.\\

\noindent I am thankful to Birgit Speh and Bent {\O}rsted for sharing their ideas on the topic of this paper. The present work is
supported by Villum Fonden (VIL54509).

\section{Convergence on spheres} \label{sec:spheres}
Let $\Gamma_n$ be the subgroup of the circle $S^1$
obtained as the image of the isomorphism mapping $[1] \in \Z/n\Z$ to a primitive root of unity $e^{2\pi\sqrt{-1}/n}$. 
The dual $\Gamma_n^* = \operatorname{Hom}(\Gamma_n,\cc^\times)$ 
is the group of characters $\chi_k$ defined by 
$
\chi_k(1) = e^{2\pi \sqrt{-1} k/n}.
$
On the space $\mathscr{C}(\Gamma_n)$ of functions on $\Gamma_n$, define an inner product of $g,h\in \mathscr{C}(\Gamma_n)$ by  
$$
\langle g, h \rangle = 
\langle g, h \rangle_{\mathscr{C}(\Gamma_n)} = \frac{1}{|\Gamma_n|} \sum_{k\in \Gamma_n} f(k)\overline{g}(k).
$$
The Fourier expansion of $g$ with respect to $\Gamma_n^\ast$ is 
$$
g = \sum_{\chi_k\in \Gamma_n^*} \langle g, \chi_k\rangle \chi_k
$$
with coefficients $\langle g,\chi_k \rangle$ given in terms of the  
Fourier transform $\widehat{g}$ as 
$$
\widehat{g}(\chi_k) = \frac{1}{|\Gamma_n|}\sum_\ell g(\ell) \overline{\chi_k}(\ell).
$$

\noindent Let $f$ be a smooth function in $\mathscr{C}^\infty(S^1)$. Supposing that $f$ is unknown but only the values on $\Gamma_n$ are given,
the Whitney extension problem for $f|_{\Gamma_n}$
is to find $W\in \mathscr{C}^\infty(S^1)$ such that 
$W(\gamma)=f(\gamma)$ for all $\gamma \in \Gamma_n$ \cite{Whitney}.
By solving a linear system of equations for $(a_\ell)\in \cc^n$, a solution takes the form 
$$
W(f|_{\Gamma_n})(\theta) =  \sum_{1 \leq  \ell \leq n} a_\ell e^{2\pi\sqrt{-1} \ell \theta },\quad \theta \in \R/\Z
$$
For an integer $j\leq n$, noting that 
$
f(j) = W(f|_{\Gamma_n})(j) =
\sum_{\ell = 1 }^n a_\ell \chi_\ell (j),
$ it follows that
$$a_\ell = \langle f|_{\Gamma_n}, \chi_\ell \rangle = \widehat{f|_{\Gamma_n}} (\ell)$$ for all $\ell \leq n$.




The identity $\chi_k(-\ell)=\chi_\ell(-k)$ gives 
$g(-\ell) = n \widehat{ \widehat{g}  } (\ell) $ and hence
for $\theta \in \R/\Z$ that
\begin{align*}
W(\widehat{f|_{\Gamma_n}}) (\theta /n ) &= \sum_{\ell=1}^n \widehat{ \widehat{f|_{\Gamma_n}} } (\ell) e^{2\pi \sqrt{-1} (\theta/n) \ell} \\
&= \frac{1}{n} \sum_{\ell=1}^n  f|_{\Gamma_n}(-\ell) e^{2\pi \sqrt{-1} \theta \ell/n} 
\xlongrightarrow{} \widehat{f} (\theta)
\end{align*}
as $n$ tends to infinity, where
$
\widehat{f}(\theta) = \int_{0}^1 f(x) e^{-2\pi \sqrt{-1} \theta x } dx
$
is
the Fourier transform of $f \in \mathscr{C}^\infty(\R/\Z)$ with respect to the circle group.\\

\noindent Next, consider the parametrization of $S^2$ given by
$$(s_1,s_2,s_3) = (\sin \theta \cos 2\pi \phi, \sin \theta \sin 2\pi \phi, \cos \theta )$$
for $(\phi,\theta) \in [0,1] \times [0,\pi]$. 
For a given dense subset $\{\theta_k: k\geq 1\}$ of $[0,\pi]$,
define the set
$$
\Gamma_n = \{ (\phi_j, \theta_k) \in [0,1] \times [0,\pi] : \phi_j = j/n, 1 \leq  j,k \leq n \}
$$
of cardinality $n^2$.
Let
$f: S^2\to \mathbb{C}$
be a given continuous function
with $y_{jk} := f(\phi_j, \theta_k)$ 
for $(\phi_j,\theta_k)\in \Gamma_n$.
For all integers $\ell \geq 1$, the functions
$$
(s_1+is_2)^{\ell} = (\sin \theta \cos 2\pi \phi + i \sin \theta \sin 2\pi \phi)^{\ell} = (e^{2\pi i \phi} \sin \theta )^{\ell}
$$
on $S^2$ are spherical harmonics of degree $\ell \geq 1$,
given as powers of the function 
$$\varphi(\phi, \theta) = e^{2\pi i \phi} \sin \theta, \quad (\theta,t) \in [0,1]\times [0, \pi].$$
If the linear system
$$
y_{jk} = \sum_{\ell = 1}^{n^2} a_{\ell,n} \varphi( \phi_j, \theta_k )^{\ell}
$$
admits a solution $(a_{\ell,n})_{\ell} \in \mathbb{C}^{n^2}$,
then we note that 
\begin{align*}
W(f|_{\Gamma_n})(\phi, \theta) &
:= \sum_{\ell=1}^{n^2} a_{\ell,n} ( e^{2\pi i \phi} \sin \theta )^{\ell}
\end{align*}
solves the Whitney extension problem for $f|_{\Gamma_n}$. 
For integers $0< \ell \leq n$,
define $\chi_{\ell,n}( j ) = e^{2\pi i  j \ell/n}$
for $j \leq n$
and note that  
$
W(f|_{\Gamma_n} )(\phi_j , \theta) 
    = \sum_{\ell=1}^{n^2} a_{\ell,n} (\sin \theta)^\ell 
    \chi_{\ell,n}(j)$.
Fixing $k, \ell \geq 1$, 
we find for all $n \geq k,\ell$ that 
\begin{align*}
a_{\ell,n} (\sin (\theta_k) )^{\ell}  &= \left\langle W(f|_{\Gamma_n} )( \cdot , \theta_k) , \chi_{\ell,n} \right\rangle_{\mathscr{C}(\Z/n\Z)} \\
&= \frac{1}{n}\sum_{j=1}^n W(f|_{\Gamma_n} )(j/n , \theta_k) e^{-2\pi i j \ell/n} \\
&= \frac{1}{n}\sum_{j=1}^n f( j/n , \theta_k ) e^{-2 \pi i j \ell/n } 
\end{align*}
Let $e_\ell (\phi) = e^{2\pi i \ell \phi}$ for $\phi \in \R/\Z$ and $\ell \in \Z$. 
If $\langle \cdot, \cdot \rangle|_{L^2(S^1)}$
is the inner product on $L^2(S^1)$,
we evaluate the limit 
\begin{align*}
\lim_{n \to \infty} a_{\ell,n} (\sin (\theta_k) )^{\ell}
&= \int_0^1 f( \phi, \theta_k) e^{ -2 \pi i \ell \phi } d\phi = \langle f(\cdot,\theta_k),e_\ell\rangle|_{L^2(S^1)}
\end{align*}
by recognizing the expression as a limit of appropriate  
Riemann sums. 
Expanding $\phi \mapsto f(\phi , \theta)$ in terms of the orthonormal basis $\{e_\ell(\phi) := e^{2\pi i \ell \phi} : \ell \in \Z \}$ of $L^2(S^1)$ gives 
$$
f( \cdot, \theta ) = \sum_{\ell=1}^\infty \langle f(\cdot ,\theta), e_\ell\rangle  e_\ell.
$$
For $
W(f|_{\Gamma_n}) (\cdot,\theta)= \sum_{\ell=1}^n a_{\ell,n} 
(\sin \theta)^\ell e_\ell  
$ it follows for all $n \geq \ell$ that  
$$
\langle W(f|_{\Gamma_n})(,\theta_k) - f(,\theta_k), e_\ell \rangle_{L^2(S^1)}
= 
\langle a_{\ell,n} (\sin \theta_k)^\ell - f(\cdot, \theta_k), e_\ell \rangle_{L^2(S^1)}.
$$
For every fixed $k,\ell$, we conclude by sending $n$ to infinity that 
$$ \lim_{n \to \infty}
\langle W(f|_{\Gamma_n})(\cdot ,\theta_k) 
- f(\cdot ,\theta_k) , e_\ell\rangle_{L^2(S^1)} = 0.
$$

Let $\theta \in [0,\pi]$ be arbitrary. Since $\{\theta_k: k\geq 1\}$ is 
assumed to be dense in 
$[0,\pi]$, choose a subsequence
tending to $\theta$ and, to ease the notation, denote the convergent subsequence by $(\theta_k)_k$. 
For every $n,\ell \geq 1$, the triangle inequality gives
\begin{align*}
|\langle W(f|_{\Gamma_n})(\cdot ,\theta) 
- f(\cdot ,\theta), e_\ell \rangle_{L^2(S^1)} | & \leq 
| \langle W(f|_{\Gamma_n})(\cdot ,\theta)-W(f|_{\Gamma_n})(\cdot ,\theta_k) ,e_\ell \rangle_{L^2(S^1)} | \\
&+ |\langle W(f|_{\Gamma_n})(\cdot ,\theta_k)-f(\cdot , \theta_k) ,e_\ell \rangle_{L^2(S^1)}| \\
&+ |\langle f(\cdot,  \theta_k) - f(\cdot ,\theta), e_\ell \rangle_{L^2(S^1)}|.
\end{align*}
Consequently, and due to  
continuity
of $\theta \mapsto f(\cdot, \theta)$ and $\theta \mapsto W(f|_{\Gamma_n})(\cdot ,\theta)$
as well as the dominated convergence theorem, 
\begin{align*}
|\langle W(f|_{\Gamma_n})(\cdot ,\theta) 
- f(\cdot ,\theta), e_\ell\rangle_{L^2(S^1)} | & \leq 
\limsup_k | \langle W(f|_{\Gamma_n})(\cdot ,\theta)-W(f|_{\Gamma_n})(\cdot ,\theta_k) ,e_\ell\rangle_{L^2(S^1)} | \\
&+ \limsup_k|\langle W(f|_{\Gamma_n})(\cdot ,\theta_k)-f(\cdot ,\theta_k) ,e_\ell\rangle_{L^2(S^1)}| \\
&+ \limsup_k|\langle f(\cdot,  \theta_k) - f(\cdot ,\theta), e_\ell\rangle_{L^2(S^1)}| \\
&= \limsup_{k}|\langle W(f|_{\Gamma_n})(\cdot ,\theta_k)-f(\cdot ,\theta_k) ,e_\ell\rangle_{L^2(S^1)}| \\
&\leq  \limsup_{k} \limsup_{m} 
|\langle W(f|_{\Gamma_m})(\cdot , \theta_k)-f(\cdot ,\theta_k) ,e_\ell\rangle_{L^2(S^1)}| \\
&= \limsup_k \lim_{m \to \infty} 
|\langle W(f|_{\Gamma_m})(\cdot ,\theta_k)-f(\cdot ,\theta_k) ,e_\ell\rangle_{L^2(S^1)}| \\
&= 0.
\end{align*}
For a family $(\varphi_n)_n$ in $L^2(S^1)$,
recall that if 
$\langle \varphi_n,e_\ell\rangle \to 0$ as $n\to \infty$ for all members $e_\ell$ of the basis, 
then 
$\langle \varphi_n, \varphi\rangle \to 0$ as $n\to \infty$
for every $\varphi \in L^2(S^1)$. 
We conclude that 
$$||W(f|_{\Gamma_n})(\cdot ,\theta)   
-f(\cdot ,\theta)  ||^2_{L^2(S^1)} = 
\int_{0}^1 |W(f|_{\Gamma_n})(\phi,\theta)   
-f(\phi ,\theta)|^2 d\phi
\to 0$$ as $n \to \infty$,
from which we derive two consequences.
First, the dominated convergence theorem implies that 
$$\int_0^{\pi} 
\int_{0}^1 |W(f|_{\Gamma_n})(\phi ,\theta)   
-f(\phi ,\theta)|^2  d\phi \sin \theta d \theta
\to 0$$ as $n \to \infty$,
so with respect to the $\operatorname{SO}(3)$-invariant measure $\sin\theta d\theta d\phi$ on $S^2$,
the family $W(f|_{\Gamma_n})$ converges to $f$ in the Hilbert space $L^2(S^2, \sin \theta d\theta d\phi)$.
Second, the limit 
\[\lim_{n \to \infty} W(f|_{\Gamma_n}) (\phi, \theta) = f(\phi,\theta)\]
holds pointwise for almost all $(\phi, \theta) \in [0,1] \times [0,\pi]$ except possibly 
a Lebesgue null set. 
However, the exceptional set will not contain the discrete space $\Gamma = \{ (\phi_j,\theta_k): j,k \geq 1 \}$, since for every $j,k \geq 1$, 
$W(f|_{\Gamma_n})(\phi_j, \theta_k) = f(\phi_j, \theta_k)$ 
eventually in $n$.\\

\noindent 
The arguments for $S^2$ extend readily to $S^{d-1}$ for all $d \geq 3$, which we briefly sketch. 
Let $f \in \mathscr{C}(S^{d-1})$ be a given continuous function and let
$$
\Gamma_n = \{ (\theta_{k_1},\ldots,\theta_{k_{d-1}}) : 
\theta_{k_1} = k_1/n, 1\leq k_j \leq n, 1\leq j \leq d-1\}
$$
parametrize a set of points in $S^{d-1}$ 
of cardinality $n^{d-1}$.
The Whitney extension of $f|_{\Gamma_n}$ 
can be chosen as a linear combination of spherical harmonics on $S^{d-1}$ given by  
$$
W(f|_{\Gamma_n})(\theta_1,\ldots,\theta_{d-1}) = \sum_{\ell=1}^n
a_{\ell,n} \left( e^{i\theta_1} \prod_{j=2}^{d-1} \sin \theta_j \right)^{\ell}.
$$
For $\chi_{\ell,n} (j) = e^{i \ell j /n}$,
\begin{align*}
    a_{\ell,n} \left( \prod_{j=2}^{d-1} \sin \theta_{k_j} \right)^{\ell}
    &= \langle W(f|_{\Gamma_n})(\cdot, \theta_{k_2},\ldots,\theta_{k_{d-1}}), \chi_{\ell,n} \rangle_{\mathscr{C}(\Z/n\Z)} \\
    &= \frac{1}{n}\sum_{j=1}^{n} W(f|_{\Gamma_n}) (j/n, \theta_{k_2},\ldots, \theta_{k_{d-1}}) e^{-i\ell j/n} \\
    &\to \int_{0}^1 f(\theta_1, \theta_{k_2},\ldots,\theta_{k_{d-1}}) e^{-i \ell \theta_1} d\theta_1
\end{align*}
as $n\to \infty$.
Reasoning as in the case of $S^2$, one deduces that 
$$
\lim_{n \to \infty}W(f|_{\Gamma_n}) (\theta_1,\ldots,\theta_{d-1}) = f (\theta_1,\ldots,\theta_{d-1})
$$
pointwise for all $(\theta_1,\ldots,\theta_{d-1})$ except perhaps on a null set in $[0,1]\times [0,\pi]^{d-2}$ not meeting
$\Gamma_n$. Moreover, $W(f|_{\Gamma_n})$ converges
to $f$ in $L^2(S^{d-1},\mu_d)$ with respect to the $\operatorname{SO}(d)$--invariant measure $\mu_d$ on 
the $(d-1)$--sphere in $\R^d$.

\section{Convergence on homogeneous spaces} \label{sec:coherent}
\begin{lemma} \label{pointwise}
Let $X$ be a homogeneous space and let
$\mathscr{H}(X)$ be a Hilbert space of functions on $X$
with a positive definite reproducing kernel 
\[
K: X \times X \to \mathbb{C}.
\] 
If a sequence $(f_k)$ of functions $f_k\in \mathscr{H}(X)$ is Cauchy in the norm topology,
then for every $z\in X$ the sequence 
$(f_k(z))$ is convergent 
as $k$ tends to infinity.
Moreover, if
$e_z:= K(\cdot,z) \in \mathscr{H}(X)$ have norms $||e_z||$
bounded locally in $z$, then 
$(f_k)$
converges locally uniformly in $z$ as $k$ tends to infinity. 
\end{lemma}
\begin{proof}
For every $z\in X$ and every $k$,
the lemma follows from the reproducing property 
\[
f_k(z) = \langle f_k, e_z \rangle
\]
for all $z\in X$.
\end{proof}
For symmetric spaces $X=G/K$ 
where $G$ is Lie group of Hermitian type and $K$ is a maximal compact subgroup, 
the holomorphic discrete series of $G$ 
can be realized on reproducing kernel spaces of holomorphic functions, and the natural topology for their convergence
is the topology of locally uniform convergence. \\


Let $\mathscr{H}(X)$ be a Hilbert space of functions on a homogeneous space $X$. 
Let $\Gamma = \{\gamma_k : k \geq 1\}$ be a discrete subspace of $X$ and set $\Gamma_n = \{\gamma_k: 1\leq k \leq n\}$.
For a given function $f: \Gamma \to \mathbb{C}$,
we will say that a function $W(f|_{\Gamma_n})$ is 
a Whitney extension of $f|_{\Gamma_n}$ 
if $W(f|_{\Gamma_n}) \in \mathscr{H}(X)$
and $W(z) = f(z)$ for all $z\in \Gamma_n$.\\

Let $\mathscr{H}$ be a Hilbert space. 
A \emph{frame} is a sequence $(e_\gamma)_{\gamma \in \Gamma}$ in 
$\mathscr{H}$ for which there exist constants $\alpha, \beta >0$ such that  
\[
\alpha ||f||^2 \leq \sum_{\gamma \in \Gamma} | \langle f, e_\gamma \rangle |^2 \leq \beta ||f||^2 
\]
for all $f\in \mathscr{H}$.
Linear dependence between the vectors in a frame is allowed.
If $(e_z)_{z\in \Gamma}$ is a frame in a Hilbert space, then the linear span of $\{e_\gamma: \gamma \in \Gamma\}$
is a dense subspace of $\mathscr{H}$. \\

Frames form a class of examples
covered by Theorem \ref{conv} below. 
In \cite{Ortega-Cerda-Seip}, the authors give sufficient conditions
on $\Gamma \subseteq (-1,1)$ for the sequences $(e^{i \gamma x})_{\gamma \in \Gamma}$ to be frames in $L^2(-1,1)$. 
Note, however, that 
the frame property is not necessary in Theorem \ref{conv}. Indeed, 
the proof works as long as $\Gamma$ contains  
enough points to allow a dense linear span of $e_\gamma$ as $\gamma$ traverses $\Gamma$.

\begin{theorem}\label{conv}
Let $X$ be a homogeneous space with respect to a reductive Lie group $G$ and 
let $\mathscr{H}(X)$ be a Hilbert space of $\mathbb{C}$-valued functions on $X$
admitting a positive definite reproducing kernel 
\[
K: X\times X \to \mathbb{C}.
\]
Let $\Gamma$ be discrete topological subspace of $X$
with a given function $f:\Gamma \to \mathbb{C}$.
Fix an enumeration of $\Gamma$ and define  
$\Gamma_n = \{z_k \in \Gamma: 1\leq k\leq n\}$ for $n\geq 1$.
If the coherent states $(e_z)_{z \in \Gamma}$
span a dense subspace of $\mathscr{H}(X)$, then 
\[
W(f|_{\Gamma_n}) \xlongrightarrow{} f
\]
in the norm topology and pointwise as $n$ tends to infinity.
\end{theorem}

\begin{proof}
The positive definiteness of $K$ implies that the matrices $(K(z_j, z_k))_{j,k}$
are invertible or, equivalently, 
that the states $e_{z_k}$ are linearly independent 
for all $z_k\in X$.
Let \[
P_n: \mathscr{H}(X) \to \operatorname{span}_{\mathbb{C}}\{ e_z: z\in \Gamma_n \} \]
be the orthogonal projection onto the closed subspace
spanned by 
$\{e_z:  z\in \Gamma_n \}$. 
For every $z \in \Gamma_n$, the vector 
$f - P_n f$ is orthogonal  
to $e_{z}$ in $\mathscr{H}(X)$ and therefore
$\langle P_n f , e_{z} \rangle
 = \langle f , e_{z} \rangle$. We conclude that $W(f|_{\Gamma_n}) = P_n f$ solves the Whitney extension problem for $f|_{\Gamma_n}$ in $\mathscr{H}(X)$.
Since the linear span of $\{e_z: z\in \Gamma\}$ is dense in $\mathscr{H}(X)$,  
it follows that  
\[
W(f|_{\Gamma_n}) = P_n f \xlongrightarrow{} f
\]
in the norm topology 
as $n\to \infty$. By Lemma \ref{pointwise}, we conclude that
$W(f|_{\Gamma_n})(z)$ converges to $f(z)$ at every $z\in X$ as $n$ tends to infinity.
\end{proof}

\bibliographystyle{amsplain} 
\bibliography{references.bib}

\nocite{*}

\end{document}